\documentclass[10pt]{article}
\font\smallit=cmti10

\usepackage{amssymb,latexsym,amsmath,epsfig,amsthm,url,mathrsfs,tikz,tkz-graph,caption,mathrsfs}
\usetikzlibrary{arrows}
\usetikzlibrary{shapes}
\usetikzlibrary{automata}

\makeatletter

\renewcommand\section{\@startsection {section}{1}{\z@}
{-30pt \@plus -1ex \@minus -.2ex}
{2.3ex \@plus.2ex}
{\normalfont\normalsize\bfseries\boldmath}}

\renewcommand\subsection{\@startsection{subsection}{2}{\z@}
{-3.25ex\@plus -1ex \@minus -.2ex}
{1.5ex \@plus .2ex}
{\normalfont\normalsize\bfseries\boldmath}}

\renewcommand{\@seccntformat}[1]{\csname the#1\endcsname. }

\makeatother

\newtheorem{theorem}{Theorem}

\newtheorem{corollary}{Corollary}

\theoremstyle{definition}
\newtheorem{definition}{Definition}

\newtheorem{example}{Example}

\begin{document}

\begin{center}
\uppercase{\bf A Multigraph Characterization of Permutiple Strings}
\vskip 20pt
{\bf Benjamin V. Holt}\\
{\smallit Department of Mathematics, Southwestern Oregon Community College, Oregon}\\
{\tt benjamin.holt@socc.edu}\\
\vskip 10pt
\end{center}
\vskip 20pt

\vskip 30pt

\centerline{\bf Abstract}

\noindent
A permutiple is a natural number whose representation in some base is an integer multiple of a number whose representation has the same collection of digits. A previous paper utilizes a finite-state-machine construction and its state graph to recognize permutiples and to generate new examples. Permutiples are associated with walks on the state graph which necessarily satisfy certain conditions. However, the above effort does not provide sufficient conditions for the existence of permutiples. In this paper, we provide such a condition which we will state using the language of multigraphs.

\pagestyle{myheadings}
\thispagestyle{empty}
\baselineskip=12.875pt
\vskip 30pt

\section{Introduction}
Let $b > 1$ be an integer. A
permutiple is a natural number which is an integer multiple of some permutation of its
digits in base $b$ \cite{holt_3}. Specific cases of digit-permutation problems include cyclic permutations of digits \cite{guttman,kalman}, for example, $714285 = 5 \cdot 142857,$ as well as digit reversals \cite{hoey, holt_1,holt_2, kendrick_1,sloane, web_wil, young_1,young_2}, which include $87912=4\cdot 21978$ and $98901 = 9 \cdot 10989.$ Numbers which are multiples of cyclic permutations of their digits are known as {\it cyclic numbers} \cite{guttman}. Numbers which are multiples of their reversals are known by several names, including {\it palintiples} \cite{hoey, holt_1,holt_2}, {\it reverse multiples} \cite{kendrick_1, sloane, young_1, young_2}, and {\it reverse divisors} \cite{web_wil}.

Another specific case of permutiples worth mentioning is a paper by Qu and Curran \cite{qu} which considers base-$b$ numbers which are multiples of $(b^{b-1}-1)/(b-1)^2,$ whose representation is all consecutive base-$b$ digits 1 through $b-1.$ Two base-10 examples include $987654312=8 \cdot 123456789$ and $493827156=4 \cdot 123456789.$

Some works place no restrictions on the type of permutation which may arise in an arbitrary base and multiplier \cite{holt_3,holt_4,holt_5, holt_6}. In \cite{holt_3,holt_4}, the author describes methods for finding new examples of permutiples from the digits of known examples. For instance, using these methods, we are able to find new examples $79128=4\cdot 19782$ and $78912=4\cdot 19728$ from the known example mentioned above, $87912=4\cdot 21978$. Other works by the author \cite{holt_5,holt_6} utilize graph-theoretical and finite-state-machine methods to produce examples with any suitable number of digits. These methods are modifications of the work of Hoey \cite{hoey} and Sloane \cite{sloane}, and both apply what are ultimately similar techniques to the digit-reversal problem. The former uses finite-state-machine methods, while the latter uses a graph-theoretical approach. Another example of applying finite-state-machine techniques to base-dependent integer sequences is a paper by Faber and Grantham \cite{faber} which finds pairs of integers whose sum is the reverse of their product. For example, $3$ and $24$ have this property since $3+24=27$ and $3\cdot 24=72.$

The methods developed in \cite{holt_5} use a finite-state-machine construction, known as the {\it Hoey-Sloane machine}, and its state graph, known as the {\it Hoey-Sloane graph}, which describes a collection of possible base-$b$ multiplications by a single-digit multiplier $n.$ The states of the Hoey-Sloane machine are the possible carries which may occur when performing digit-preserving multiplication, and the input alphabet consists of ordered pairs representing directed edges from the {\it mother graph}, which catalogs how digits may be permuted in a single-digit multiplication. The initial state of the machine must be zero since the carry corresponding to the first digit in any single-digit multiplication is defined to be zero \cite{holt_3}. Similarly, the terminal digit of a multiplication requires the final (accepting) state to be zero, otherwise there would be digits beyond the terminal digit. It is shown in \cite{holt_5} that input strings which represent permutiples, known as {\it permutiple strings}, consist of ordered pairs which make up cycles on the mother graph. In this context, permutiples correspond to a sequence of state transitions (walks on the Hoey-Sloane graph) beginning and ending with the zero state, where each transition is induced by a collection of edges which is a multiset union of cycles of the mother graph.

In this paper, we continue the work described above to establish sufficient conditions for the existence of permutiple strings, yielding a multigraph characterization of permutiple strings. We also apply the results to several examples, which describe how to create novel permutiple examples of any suitable length.

\section{Summary of Previous Work}

What follows is a summary of several results and defnitions from previous works \cite{holt_3,holt_5} that we will use in this article.

\subsection{Basic Definitions and Results}

The notation $(d_k,d_{k-1},\ldots,d_0)_b$ is used to represent $\sum_{j=0}^{k}d_j b^j,$ where $0\leq d_j<b$ for all $0\leq j \leq k.$ We may now define what it means to be a permutiple number.

\begin{definition}[\cite{holt_3}]
Let $1<n<b$ be a natural number, and let $\sigma$ be a permutation on $\{0,1,2,\ldots, k\}$.
We say that $(d_k, d_{k-1},\ldots, d_0)_b$  is an $(n,b,\sigma)$-\textit{permutiple} provided
\[
(d_k,d_{k-1},\ldots,d_1, d_0)_b=n(d_{\sigma(k)},d_{\sigma(k-1)},\ldots, d_{\sigma(1)}, d_{\sigma(0)})_b.
\]
When $\sigma$ itself is not important to the discussion, we will refer to $(d_k,d_{k-1},\ldots,d_0)_b$ as simply an $(n,b)$-permutiple. The collection of all base-$b$ permutiples having multiplier $n$ will be referred to as {\it $(n,b)$-permutiples}.
\end{definition}

The next result relates the digits and carries of a permutiple.

\begin{theorem}[\cite{holt_3}]
Let $(d_k, d_{k-1},\ldots, d_0)_b$ be an $(n,b,\sigma)$-permutiple, and let $c_j$ be the $j^{th}$carry. Then,
\[
b c_{j+1}-c_j=nd_{\sigma(j)}-d_{j}
\]
for all $0\leq j \leq k$.
\label{digits_carries_1}
\end{theorem}

The carries in any permutiple multiplication are always less than the multiplier.

\begin{theorem}[\cite{holt_3}] \label{carries}
Let $(d_k, d_{k-1},\ldots, d_0)_b$ be an $(n,b,\sigma)$-permutiple, and let $c_j$ be the
$j^{th}$ carry. Then, $c_j\leq n-1$ for all $0 \leq j \leq k$.
\end{theorem}

\subsection{Permutiple Graphs and the Mother Graph}

For any permutiple, we may define a directed graph which describes the most essential properties of the digit permutation.

\begin{definition}[\cite{holt_5}]
Let $p=(d_k, d_{k-1},\ldots, d_0)_b$  be an $(n,b,\sigma)$-permutiple. We define a directed graph, called the {\it graph of $p$}, denoted by $G_p,$ to consist of the collection of base-$b$ digits as vertices, and the collection of directed edges $E_p=\left \{\left (d_j,d_{\sigma(j)}\right) |\, 0 \leq j \leq k \right \}.$  A graph $G$ for which there is a permutiple $p$ such that $G=G_p$ is called a {\it permutiple graph}.
\label{class_graph}
\end{definition}

For the remainder of this paper, we may drop the ``directed'' terminology with the understanding that all graphs and multigraphs considered here will have directed edges.

Table \ref{conj_class_table} gives a collection of permutiples with the same graph. Their common graph is shown in Figure \ref{conj_class_graph}.

\begin{center}
 \begin{tabular}{|c|c|}
\hline $(4,10,\sigma)$-Permutiple &  $\sigma$ \\\hline
$(8,7,9,1,2)_{10}=4 \cdot (2,1,9,7,8)_{10}$ &  $\rho$  \\\hline
$(8,7,1,9,2)_{10}=4 \cdot (2,1,7,9,8)_{10}$ & $(1,2)\rho(1,2)$ \\\hline
$(7,9,1,2,8)_{10}=4 \cdot (1,9,7,8,2)_{10}$ & $\psi^{-4}\rho\psi^4$ \\\hline
$(7,1,9,2,8)_{10}=4 \cdot (1,7,9,8,2)_{10}$ & $\psi^{-4}(1,2)\rho(1,2)\psi^4$ \\\hline
\end{tabular}
\captionof{table}{Permutiples with the same graph as $(8,7,9,1,2)_{10}=4 \cdot (2,1,9,7,8)_{10},$ where $\psi=(0,1,2,3,4)$ and $\rho$ is the reversal permutation.}\label{conj_class_table}
\end{center}

\begin{center}
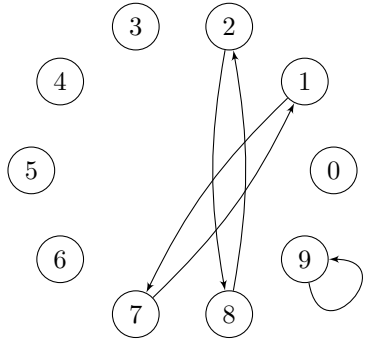

\begin{tikzpicture}
\tikzset{edge/.style = {->,> = latex'}}
\tikzset{vertex/.style = {shape=circle,draw,minimum size=1.5em}}
[xscale=3, yscale=3, auto=left,every node/.style={circle,fill=blue!20}]
\node[vertex] (n0) at (12,5) {$0$};
\node[vertex] (n1) at (11.618,6.17557) {$1$};
\node[vertex] (n2) at (10.618,6.90211) {$2$};
\node[vertex] (n3) at (9.38197,6.90211) {$3$};
\node[vertex] (n4) at (8.38197,6.17557) {$4$};
\node[vertex] (n5) at (8,5.00001) {$5$};
\node[vertex] (n6) at (8.38196,3.82443) {$6$};
\node[vertex] (n7) at (9.38196,3.09789) {$7$};
\node[vertex] (n8) at (10.618,3.09788) {$8$};
\node[vertex] (n9) at (11.618,3.82442) {$9$};
\draw[edge, bend right=10] (n8) to (n2);
\draw[edge, bend right=10] (n2) to (n8);
\draw[edge, bend right=10] (n7) to (n1);
\draw[edge, bend right=10] (n1) to (n7);
\draw[edge] (n9) to[in=0,out=-80, loop, style={min distance=10mm}] (n9);
\end{tikzpicture}
\captionof{figure}{The directed graph which results from taking the collection of ordered pairs $\left\{\left (d_j,d_{\sigma(j)}\right) |\, 0 \leq j \leq 4 \right\}$ from any example in Table \ref{conj_class_table} as edges.}
\label{conj_class_graph}
\end{center}

Permutiple graphs provide a framework for classifying permutiples.

\begin{definition}[\cite{holt_5}]\label{perm_class}
Let $p$ be an $(n,b)$-permutiple with graph $G_p.$ We define the {\it class of $p$} to be the collection $C$ of all $(n,b)$-permutiples $q$ such that $G_q$ is a subgraph of $G_p.$ We also define the graph of the class to be $G_p,$ which we will denote as $G_C$ and will call the {\it graph of $C.$}
\end{definition}

For an $(n,b)$-permutiple $(d_k,\ldots, d_0)_b=n(d_{\sigma(k)},\ldots, d_{\sigma(0)})_b$ it is shown in \cite{holt_5} that $\lambda(d_j+(b-n)d_{\sigma(j)})\leq n-1$ for all $0\leq j\leq k,$ where $\lambda$ gives the least non-negative residue modulo $b.$ This condition puts a restriction on the possible edges of a permutiple graph, and we state it as a theorem.

\begin{theorem}[\cite{holt_5}]\label{possible_edges_mg}
 Let $p=(d_k, d_{k-1},\ldots, d_0)_b$ be an $(n,b,\sigma)$-permutiple with graph $G_p.$ Then, for any edge $(d_j,d_{\sigma(j)})$ of $G_p,$ it must be that $\lambda \left(d_j+(b-n)d_{\sigma(j)}\right)\leq n-1$
for all $0\leq j\leq k,$ where $\lambda$ gives the least non-negative residue modulo $b.$
\end{theorem}

Theorem \ref{possible_edges_mg} enables us to gather all possible edges of a permutiple graph into a single graph to obtain the {\it mother graph}.

\begin{definition}[\cite{holt_5}]\label{mother_graph}
The $(n,b)$-\textit{mother graph}, denoted $M,$ is the graph having all base-$b$ digits as its vertices and the collection of edges $(d_1,d_2)$ satisfying the inequality $\lambda \left(d_1+(b-n)d_{2}\right)\leq n-1.$
\end{definition}

The next result underpins the methods presented in \cite{holt_5}.

\begin{theorem}[\cite{holt_5}]
Let $C$ be any $(n,b)$-permutiple class. Then, $G_C$ is a union of cycles of $M.$
\label{cycle_union}
\end{theorem}

We now provide the reader with an example which brings together the concepts we have covered so far.

\begin{example}\label{example_1}
The $(4,10)$-mother graph is displayed in Figure \ref{4_10_mg}. Letting $p=(8,7,9,1,2)_{10}=4 \cdot (2,1,9,7,8)_{10}$ from Table \ref{conj_class_table}, and $C$ be the $(4,10)$-permutiple class with graph $G_C=G_p,$ Figure \ref{4_10_mg} features the graph of $C$ in bold red.

\begin{center}
\begin{tikzpicture}
\tikzset{edge/.style = {->,> = latex'}}
\tikzset{vertex/.style = {shape=circle,draw,minimum size=1.5em}}
[xscale=2, yscale=2, auto=left,every node/.style={circle,fill=blue!20}]

\node[vertex] (n0) at (13,5) {$0$};
\node[vertex,red, very thick] (n1) at (12.427051918969068,6.763354468797628) {$\bf 1$};
\node[vertex,red, very thick] (n2) at (10.927054011580958,7.853168564878643) {$\bf 2$};
\node[vertex] (n3) at (9.07295355956129,7.853171024889661) {$3$};
\node[vertex] (n4) at (7.5729527602588975,6.76336090919162) {$4$};
\node[vertex] (n5) at (7.000000000010562,5.00000796076938) {$5$};
\node[vertex] (n6) at (7.572943401820056,3.236651971608781) {$6$};
\node[vertex,red, very thick] (n7) at (9.072938417283323,2.1468338951524664) {$\bf 7$};
\node[vertex,red, very thick] (n8) at (10.927038869289936,2.1468265151194124) {$\bf 8$};
\node[vertex,red, very thick] (n9) at (12.427042560496046,3.236632650426805) {$\bf 9$};
\draw[edge](n0) to[in=40,out=-40, loop, style={min distance=10mm}] (n0);
\draw[edge, bend right=5](n0) to (n2);
\draw[edge, bend right=0](n0) to (n5);
\draw[edge, bend right=5](n0) to (n7);
\draw[edge, bend right=0](n1) to (n0);
\draw[edge, bend right=0](n1) to (n2);
\draw[edge, bend right=5](n1) to (n5);
\draw[edge, red, bend right=5, very thick](n1) to (n7);
\draw[edge, bend right=5](n2) to (n0);
\draw[edge, bend right=0](n2) to (n3);
\draw[edge, bend right=0](n2) to (n5);
\draw[edge, red, bend right=5, very thick](n2) to (n8);
\draw[edge, bend right=0](n3) to (n0);
\draw[edge](n3) to[in=140,out=60, loop, style={min distance=10mm}] (n3);
\draw[edge, bend right=5](n3) to (n5);
\draw[edge, bend right=0](n3) to (n8);
\draw[edge, bend right=0](n4) to (n1);
\draw[edge, bend right=0](n4) to (n3);
\draw[edge, bend right=5](n4) to (n6);
\draw[edge, bend right=5](n4) to (n8);
\draw[edge, bend right=5](n5) to (n1);
\draw[edge, bend right=5](n5) to (n3);
\draw[edge, bend right=0](n5) to (n6);
\draw[edge, bend right=0](n5) to (n8);
\draw[edge, bend right=0](n6) to (n1);
\draw[edge, bend right=5](n6) to (n4);
\draw[edge](n6) to[in=190,out=270, loop, style={min distance=10mm}] (n6);
\draw[edge, bend right=0](n6) to (n9);
\draw[edge, red, bend right=5,very thick](n7) to (n1);
\draw[edge, bend right=0](n7) to (n4);
\draw[edge, bend right=0](n7) to (n6);
\draw[edge, bend right=5](n7) to (n9);
\draw[edge, red, bend right=5, very thick](n8) to (n2);
\draw[edge, bend right=5](n8) to (n4);
\draw[edge, bend right=0](n8) to (n7);
\draw[edge, bend right=0](n8) to (n9);
\draw[edge, bend right=0](n9) to (n2);
\draw[edge, bend right=0](n9) to (n4);
\draw[edge, bend right=5](n9) to (n7);
\draw[edge,red, very thick](n9) to[in=0,out=-80, loop, style={min distance=10mm}] (n9);
\end{tikzpicture}

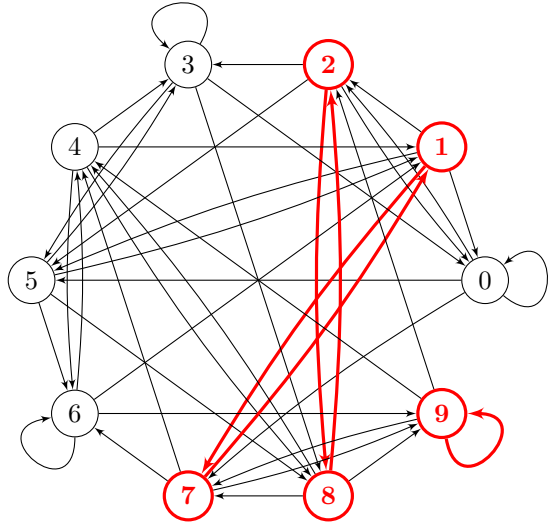
\captionof{figure}{The $(4,10)$-mother graph with the graph of $C$ featured in bold red.}
\label{4_10_mg}
\end{center}
\end{example}

\subsection{Finite-State-Machine Description of the Permutiple Problem}

Taking the carries of a permutiple as a collection of states, the above concepts can be placed within a finite-state-machine framework. Taking non-negative integers less than $n$ as the collection of states, and the edges of $M$ as the input alphabet, the equation
\begin{equation}
 c_2=\left[nd_{2}-d_1+c_1\right]\div b
 \label{state_transition}
\end{equation}
defines a state-transition function from state $c_1$ to state $c_2$ with $(d_1,d_{2})$ serving as the input which induces the transition. This transition corresponds to a labeled edge on the state diagram as seen in Figure \ref{state_diagram}.
\begin{center}
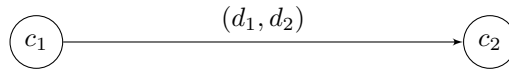

\begin{tikzpicture}
\tikzset{edge/.style = {->,> = latex'}}
\tikzset{vertex/.style = {shape=circle,draw,minimum size=1.5em}}
[xscale=2, yscale=2, auto=left,every node/.style={circle,fill=blue!20}]
\node[vertex] (n0) at (7,5) {$c_1$};
\node[vertex] (n1) at (13,5) {$c_2$};
\draw[edge] (n0) edge node[above] {$(d_1,d_2)$} (n1);
\end{tikzpicture}
\captionof{figure}{An edge on the state diagram.}
\label{state_diagram}
\end{center}
The first carry $c_0$ of any single-digit multiplication is zero by definition \cite{holt_3}. This is to say that the initial state must be zero. Also, for an $\ell$-digit $(n,b)$-permutiple, $c_{\ell}$ must also be zero, otherwise, the result would be an $(\ell+1)$-digit number. Thus, the zero state is the only possible accepting state. The above is called the $(n,b)$-{\it Hoey-Sloane machine}, and its state diagram is called the $(n,b)$-{\it Hoey-Sloane graph}, which we denote as $\Gamma.$

Digit pairs $(d_1,d_{2})$ which solve Equation (\ref{state_transition}) for particular values of $c_1$ and $c_2$ are not unique. That is, there are generally multiple inputs that the machine will recognize for a transition to occur. Thus, a collection of inputs is assigned to each edge on $\Gamma$ by the mapping  $(c_1,c_2) \mapsto \{(d_1,d_2)\in E_M|c_2=\left[nd_{2}-d_1+c_1\right]\div b\},$ where $E_M$ is the collection of edges of $M.$

In the next section, we will define a labeled multigraph representation of $\Gamma,$ where a unique multi-edge is assigned to each input.

The language of input strings accepted by the $(n,b)$-Hoey-Sloane machine is denoted as $L.$ Thus, $L$ may be described as finite sequences of edge-label inputs which define walks on $\Gamma$ whose initial and final states are zero. Such walks are called {\it $L$-walks}. Members of $L$ which produce permutiple numbers are called {\it $(n,b)$-permutiple strings.}

We may interpret Theorem \ref{cycle_union} anew in this setting.

\begin{corollary}[\cite{holt_5}]\label{cycle_union_2}
Let $s=(d_{0},\hat{d_{0}})(d_{1},\hat{d_{1}})\cdots(d_{k},\hat{d_{k}})$ be a member of $L.$ If $s$ is a permutiple string, then the collection of ordered-pair inputs of $s$ is a union of cycles of $M.$
\end{corollary}

Corollary \ref{cycle_union_2} tells us that any permutiple string must consist of a collection of mother-graph edge inputs which induce an $L$-walk on $\Gamma$ and whose union is a collection of cycles of $M.$ We note, however, that satisfying the above two conditions is not sufficient for a string to be a permutiple string. An example is provided by \cite{holt_5} of a member of $L$ whose union is a collection of mother-graph cycles, yet is not a permutiple string. A more precise description of these ideas requires that we restate some definitions from \cite{holt_5}.

\begin{definition}[\cite{holt_5}]\label{cycle_image}
 Let $\mathscr{C}=\{C_0,C_1,\ldots,C_m\}$ be the cycles of $M.$ For each element $C_j$ of $\mathscr{C},$ define a subgraph $\Gamma_j$ of $\Gamma,$ where each edge of $\Gamma_j$ is assigned the edge-label collection by the mapping
$
(c_1,c_2)\mapsto \left\{(d_1,d_2)\in C_j|c_2=\left[nd_{2}-d_1+c_1\right]\div b\right \}.
$
Any edge $(c_1,c_2)$ for which this collection is empty will not be included as an edge on $\Gamma_j.$ With the above edges, any state for which both the indegree and outdegree are zero will not be included as a vertex. Each $\Gamma_j$ will be referred to as the {\it image of $C_j,$} or simply as a {\it cycle image}.
\end{definition}

In what follows, we will distinguish the usual set-theoretic union $\cup$ from multiset unions by using the notation $\uplus.$ For example, the multiset union of the multisets $\{1,2,2,3\}$ and $\{2,3,4\}$ is denoted as $\{1,2,2,3\} \uplus\{2,3,4\}=\{1,2,2,2,3,3,4\}.$ We suppose that $I$ is a multiset whose support is a subset $J$ of $\{0,1,\ldots,m\}.$ Then, if the cycle-image union $\Gamma_{J}=\bigcup_{j\in J}\Gamma_j$ (edge labels included) is a strongly-connected subgraph of $\Gamma$ containing the zero state, then $\Gamma_J$ describes a machine which recognizes members of $L$ whose inputs form the union of cycles $\bigcup_{j\in J}C_j.$ If the multiset cycle union $C_{I}=\biguplus_{j\in I}C_j$ can be ordered into a string $s$ belonging to $L,$ then $s$ is a permutiple string. We note that every element of $C_I$ must be used, including repeated elements, otherwise, the multisets of left and right components will not be equal, resulting in a multiplication which does not preserve the digits.

We now provide the reader with another example by considering the $(4,10)$-Hoey-Sloane graph.

\begin{example}\label{example_2}
The $(4,10)$-Hoey-Sloane graph is shown in Figure \ref{4_10_hsg}. The graph depicting the union of the images of the cycles of $G_C$ from Example \ref{example_1} is shown in bold red. The cycles of $G_C$ are $C_0=\{(9,9)\},$ $C_1=\{(2,8),(8,2)\},$ and $C_2=\{(1,7),(7,1)\},$ and the cycle-image union may be more precisely denoted as $\Gamma_0 \cup \Gamma_1 \cup \Gamma_2.$

Using the Hoey-Sloane graph, the multiset union $C_0 \uplus C_1 \uplus C_1 \uplus C_2 \uplus C_2$ may be ordered into a member of $L,$ forming a permutiple string. There are multiple ways of accomplishing this, one of which is $(8,2)(8,2)(2,8)(9,9)(1,7)(1,7)(7,1)(2,8)(7,1),$ yielding a new $(4,10)$-permutiple $(7,2,7,1,1,9,2,8,8)_{10}=4\cdot(1,8,1,7,7,9,8,2,2)_{10}.$ Another possible ordering is $(8,2)(2,8)(1,7)(7,1)(2,8)(9,9)(1,7)(7,1)(8,2),$ which gives the new example $(8,7,1,9,2,7,1,2,8)_{10}=4\cdot(2,1,7,9,8,1,7,8,2)_{10}.$

As noted in \cite{holt_6}, not every multiset union can be ordered into a permutiple string. As a trivial example, any of the above cycles individually are not sufficient to form an element of $L.$ A less trivial example is $C_1\uplus C_1 \uplus C_2,$ which is also impossible to order into a member of $L.$ We will address this specific case in a later example.

Generally speaking, we see that a strongly-connected union of cycle images containing the zero state is a necessary condition for ordering its corresponding multiset union of mother-graph cycles into permutiple strings \cite{holt_5}, but it is not a sufficient one. Finding sufficient conditions is the purpose of this effort.

\begin{center}
\begin{tikzpicture}
\tikzset{edge/.style = {->,> = latex'}}
\tikzset{vertex/.style = {shape=circle,draw,minimum size=1.5em}}
[xscale=2, yscale=2, auto=left,every node/.style={circle,fill=blue!20}]
\node[vertex,red,initial,accepting,very thick] (n0) at (0,5) {$\bf 0$};
\node[vertex] (n1) at (3.0,5) {$1$};
\node[vertex] (n2) at (6.0,5) {$2$};
\node[vertex,red,very thick] (n3) at (9.0,5) {$\bf 3$};
\draw[edge,red,very thick](n0) to[in=150,out=90, loop, style={min distance=13mm}] node[above] {\footnotesize{$\color{black}{(0,0), (4,1), }\,\color{red}{\bf (8,2)}$}} (n0);
\draw[edge](n1) to[in=110,out=70, loop, style={min distance=7mm}] node[above] {\footnotesize{$(3,3),(7,4)$}} (n1);
\draw[edge](n2) to[in=110,out=70, loop, style={min distance=7mm}] node[above] {\footnotesize{$(2,5),(6,6)$}} (n2);
\draw[edge,red,very thick](n3) to[in=90,out=30, loop, style={min distance=13mm}] node[above] {\footnotesize{$\color{red}{\bf (1,7)}\color{black}{,(5,8),} \color{red}{\bf (9,9)}$}} (n3);
\draw[edge, bend left=10](n0) edge node[above] {\footnotesize{$(2,3),(6,4)$}} (n1);
\draw[edge, bend left=10](n1) edge node[below] {\footnotesize{$(1,0),(5,1),(9,2)$}} (n0);
\draw[edge, bend left=10](n2) edge node[below] {\footnotesize{$(0,2),(4,3),(8,4)$}} (n1);
\draw[edge, bend left=10](n1) edge node[above] {\footnotesize{$(1,5),(5,6),(9,7)$}} (n2);
\draw[edge, bend left=55](n1) edge node[above] {\footnotesize{$(3,8),(7,9)$}} (n3);
\draw[edge, bend right=10](n2) edge node[below] {\footnotesize{$(0,7),(4,8),(8,9)$}} (n3);
\draw[edge, bend right=10](n3) edge node[above] {\footnotesize{$(3,5),(7,6)$}} (n2);
\draw[edge, bend left=45](n3) edge node[below] {\footnotesize{$(1,2),(5,3),(9,4)$}} (n1);

\draw[edge, bend right=45](n0) edge node[below] {\footnotesize{$(0,5),(4,6),(8,7)$}} (n2);
\draw[edge, red, bend right=60,very thick](n0) edge node[below] {\footnotesize{${\bf (2,8)}\color{black}{,(6,9)}$}} (n3);
\draw[edge, bend right=55](n2) edge node[above] {\footnotesize{$(2,0),(6,1)$}} (n0);
\draw[edge, red, bend right=70,very thick](n3) edge node[above] {\footnotesize{$\color{black}{(3,0),}\,\color{red}{{\bf (7,1)}}$}} (n0);
\end{tikzpicture}

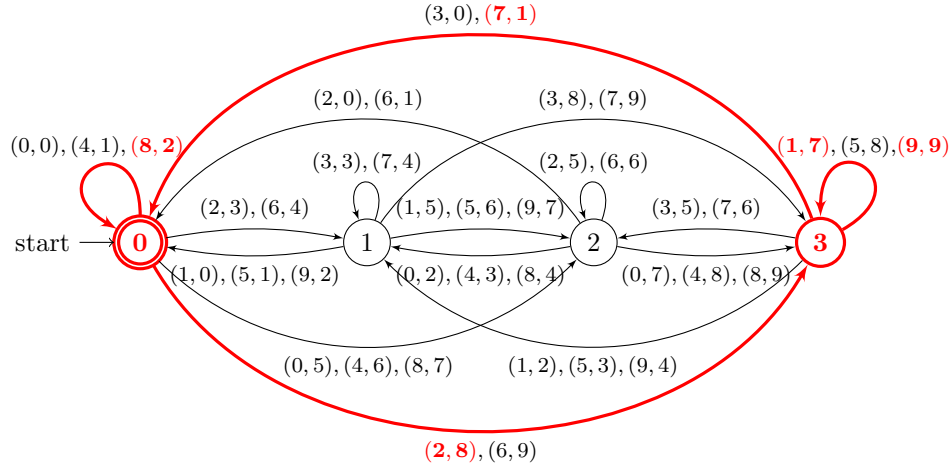
\captionof{figure}{The $(4,10)$-Hoey-Sloane graph with cycle images of $G_C$ in bold red.}
\label{4_10_hsg}
\end{center}
\end{example}

\section{The Hoey-Sloane Multigraph}

We begin by stating and proving a result which shows that an edge label $(d_1,d_2)$ cannot appear on two or more distinct edges of $\Gamma.$

\begin{theorem}\label{unique_transition}
 If $c_1,$ $c_2,$ $\hat{c}_1,$ and $\hat{c}_2$ are states on $\Gamma,$ and $(d_1,d_2)$ is an input associated with the transitions $(c_1,c_2)$ and $(\hat{c}_1,\hat{c}_2),$ then $(c_1,c_2)=(\hat{c}_1,\hat{c}_2).$
 \end{theorem}
 \begin{proof}
 By Equation (\ref{state_transition}), we have both
$bc_2-c_1=nd_2-d_1$ and  $b\hat{c}_2-\hat{c}_1=nd_2-d_1.$ Reducing both equations modulo $b,$ we have $c_1\equiv d_1-nd_2\equiv \hat{c}_1 \pmod{b}.$ Since $c_1$ and $\hat{c}_1$ are less than $n$ by definition, it follows that $c_1=\hat{c}_1.$ A routine calculation then shows that $c_2=\hat{c}_2.$
\end{proof}

 We now show that any edge $(d_1,d_2)$ of $M$ induces some transition $(c_1,c_2)$ on $\Gamma.$

\begin{theorem}\label{mg_edge_to_transition}
Let $(d_1,d_2)$ be an edge of $M.$ Then, there are integers $0\leq c_1 \leq n-1$ and $0 \leq c_2 \leq n-1$ such that $bc_2-c_1=nd_2-d_1.$
\end{theorem}

\begin{proof}
From our assumption, we know that $\lambda \left(d_1+(b-n)d_{2}\right)\leq n-1.$ Then, $d_1+(b-n)d_2\equiv  d_1-nd_2 \equiv c_1 \pmod b,$ where $0\leq c_1 \leq n-1.$ It follows that $nd_2-d_1 \equiv -c_1 \pmod b.$ We then have, for some integer $c_2,$ that $nd_2-d_1=bc_2-c_1.$ In another form, $bc_2=nd_2-d_1+c_1,$ from which we may say that $-b-1\leq bc_2 \leq  n(b-1)+n-1=nb-1.$ Thus, $0\leq c_2 \leq n-1,$ and the proof is complete.
\end{proof}

With Theorems \ref{unique_transition} and \ref{mg_edge_to_transition}, we may now define a directed, labeled multigraph which we will call the $(n,b)$-Hoey-Sloane multigraph.

\begin{definition}
Let $N$ be the collection of non-negative integers less than $n.$  We map each edge $(d_1,d_2)$ of $M$ to the multi-edge $(c_1,c_2)$ in $N\times N,$ which uniquely satisfies Equation (\ref{state_transition}), to form a labeled multi-edge, which we may visualize in the same fashion as Figure \ref{state_diagram}. Taking $N$ as the collection of vertices, along with the collection of labeled multi-edges defined above, we will call this construction the {\it $(n,b)$-Hoey-Sloane multigraph}. To distinguish the multigraph from the usual $(n,b)$-Hoey-Sloane graph $\Gamma,$ we will denote the $(n,b)$-Hoey-Sloane multigraph as $\Delta.$
\end{definition}

We note that $\Delta$ is simply an alternative representation of $\Gamma,$ where each multi-edge has a unique element from $M$ as a label, rather than a single edge with a collection of inputs from $M.$ For this reason, when referencing $\Delta,$ we will retain the finite-state-machine terminology used when referencing $\Gamma.$

We now provide two simple examples of the above definition.

\begin{example}\label{example_3}

The $(2,4)$-mother graph $M$ is seen in Figure \ref{2_4_mg}. Mapping each edge of $M$ to its multi-edge uniquely determined by Equation (\ref{state_transition}), we obtain the $(2,4)$-Hoey-Sloane multigraph $\Delta$ in Figure \ref{2_4_hsg}.

\begin{center}
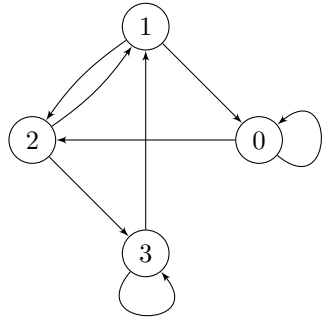

\begin{tikzpicture} \tikzset{edge/.style = {->,> = latex'}} \tikzset{vertex/.style = {shape=circle,draw,minimum size=1.5em}} [xscale=2, yscale=2, auto=left,every node/.style={circle,fill=blue!20}]
\node[vertex] (n0) at (3.5,5) {$0$}; \node[vertex] (n1) at (2.000001990192345,6.49999999999868) {$1$};
 \node[vertex] (n2) at (0.50000000000528,5.00000398038469) {$2$}; \node[vertex] (n3) at (1.999994029422965,3.500000000011883) {$3$};
\draw[edge](n0) to [in=40,out=-40, loop, style={min distance=10mm}] (n0);
\draw[edge](n0) to (n2);
\draw[edge](n1) to (n0);
\draw[edge, bend right=10](n1) to (n2);
\draw[edge, bend right=10](n2) to (n1);
\draw[edge](n2) to (n3);
\draw[edge](n3) to (n1);
\draw[edge](n3) to[in=310,out=230, loop, style={min distance=10mm}] (n3);
\end{tikzpicture}
\captionof{figure}{The $(2,4)$-mother graph.}
\label{2_4_mg}
\end{center}

\begin{center}
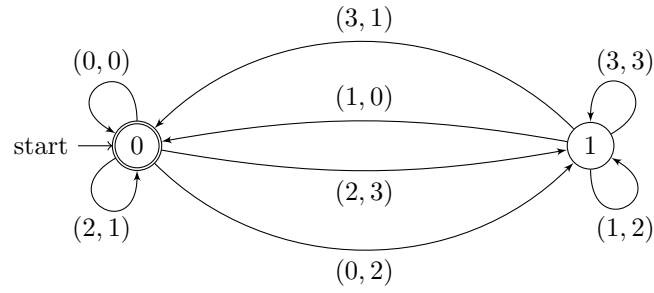

\begin{tikzpicture}
\tikzset{edge/.style = {->,> = latex'}}
\tikzset{vertex/.style = {shape=circle,draw,minimum size=1.5em}}
[xscale=2, yscale=2, auto=left,every node/.style={circle,fill=blue!20}]
\node[vertex,accepting,initial] (n0) at (7,5) {$0$};
\node[vertex] (n1) at (13,5) {$1$};
\draw[edge, bend right=10] (n0) edge node[below] {$(2,3)$} (n1);
\draw[edge, bend right=45] (n0) edge node[below] {$(0,2)$} (n1);
\draw[edge, bend right=10] (n1) edge node[above] {$(1,0)$} (n0);
\draw[edge, bend right=45] (n1) edge node[above] {$(3,1)$} (n0);
\draw[edge](n0) to[in=150,out=90, loop, style={min distance=10mm}] node[above] {$(0,0)$} (n0);
\draw[edge](n0) to[in=270,out=210, loop, style={min distance=10mm}] node[below] {$(2,1)$} (n0);
\draw[edge](n1) to[in=90,out=30, loop, style={min distance=10mm}] node[above] {$(3,3)$} (n1);
\draw[edge](n1) to[in=-30,out=-90, loop, style={min distance=10mm}] node[below] {$(1,2)$} (n1);
\end{tikzpicture}
\captionof{figure}{The $(2,4)$-Hoey-Sloane multigraph.}
\label{2_4_hsg}
\end{center}

\end{example}

\begin{example}\label{example_4}
The $(3,4)$-mother graph $M$ is displayed in Figure \ref{3_4_mg}. Mapping each edge of $M$ to its multi-edge, we obtain the $(3,4)$-Hoey-Sloane multigraph $\Delta$ shown in Figure \ref{3_4_hsg}.
\begin{center}
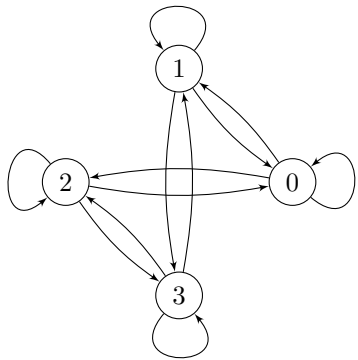

\begin{tikzpicture}
\tikzset{edge/.style = {->,> = latex'}} \tikzset{vertex/.style = {shape=circle,draw,minimum size=1.5em}} [xscale=2, yscale=2, auto=left,every node/.style={circle,fill=blue!20}] \node[vertex] (n0) at (3.5,5) {$0$}; \node[vertex] (n1) at (2.000001990192345,6.49999999999868) {$1$};
 \node[vertex] (n2) at (0.50000000000528,5.00000398038469) {$2$}; \node[vertex] (n3) at (1.999994029422965,3.500000000011883) {$3$};
\draw[edge](n0) to[in=40,out=-40, loop, style={min distance=10mm}] (n0);
\draw[edge, bend right=10](n0) to (n1);
\draw[edge, bend right=10](n0) to (n2);
\draw[edge, bend right=10](n1) to (n0);
\draw[edge](n1) to[in=130,out=50, loop, style={min distance=10mm}] (n1);
\draw[edge, bend right=10](n1) to (n3);
\draw[edge, bend right=10](n2) to (n0);
\draw[edge](n2) to[in=220,out=130, loop, style={min distance=10mm}] (n2);
\draw[edge, bend right=10](n2) to (n3);
\draw[edge, bend right=10](n3) to (n1);
\draw[edge, bend right=10](n3) to (n2);
\draw[edge](n3) to[in=310,out=230, loop, style={min distance=10mm}] (n3);
\end{tikzpicture}
\captionof{figure}{The $(3,4)$-mother graph.}
\label{3_4_mg}
\end{center}

\begin{center}
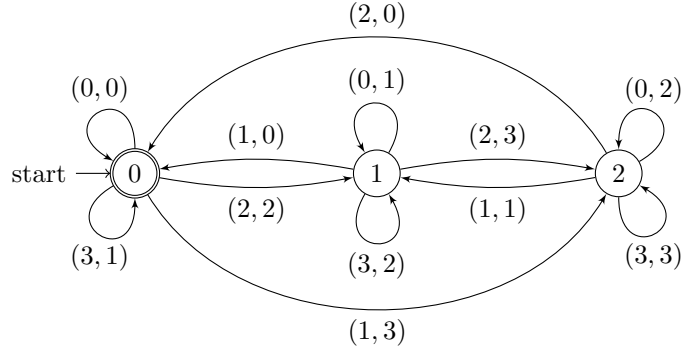

\begin{tikzpicture}
\tikzset{edge/.style = {->,> = latex'}}
\tikzset{vertex/.style = {shape=circle,draw,minimum size=1.5em}}
[xscale=2, yscale=2, auto=left,every node/.style={circle,fill=blue!20}]
\node[vertex,initial,accepting] (n0) at (0,5) {$0$};
\node[vertex] (n1) at (3.2,5) {$1$};
\node[vertex] (n2) at (6.4,5) {$2$};

\draw[edge, bend right=10](n0) edge node[below] {$(2,2)$} (n1);
\draw[edge, bend right=10](n1) edge node[above] {$(1,0)$} (n0);
\draw[edge, bend left=10](n2) edge node[below] {$(1,1)$} (n1);
\draw[edge, bend left=10](n1) edge node[above] {$(2,3)$} (n2);
\draw[edge, bend right=59](n0) edge node[below] {$(1,3)$} (n2);
\draw[edge, bend right=59](n2) edge node[above] {$(2,0)$} (n0);

\draw[edge](n0) to[in=150,out=90, loop, style={min distance=10mm}] node[above] {$(0,0)$} (n0);
\draw[edge](n0) to[in=270,out=210, loop, style={min distance=10mm}] node[below] {$(3,1)$} (n0);

\draw[edge](n1) to[in=120,out=60, loop, style={min distance=10mm}] node[above] {$(0,1)$} (n1);
\draw[edge](n1) to[in=300,out=240, loop, style={min distance=10mm}] node[below] {$(3,2)$} (n1);

\draw[edge](n2) to[in=90,out=30, loop, style={min distance=10mm}] node[above] {$(0,2)$} (n2);
\draw[edge](n2) to[in=-30,out=-90, loop, style={min distance=10mm}] node[below] {$(3,3)$} (n2);

\end{tikzpicture}
\captionof{figure}{The $(3,4)$-Hoey-Sloane multigraph.}
\label{3_4_hsg}
\end{center}

\end{example}

\section{A Multigraph Condition for Permutiple Strings}

We now define the multigraph analogue of the cycle images, which we encountered in Definition \ref{cycle_image}.

\begin{definition}\label{multi_cycle_image}
 Let $\mathscr{C}=\{C_0,C_1,\ldots,C_m\}$ be the cycles of $M.$ For each element $C_j$ of $\mathscr{C},$ define a multigraph $\Delta_j$ to consist of all edges $(c_1,c_2)$ of the $(n,b)$-Hoey-Sloane multigraph $\Delta$ for which the inputs in $C_j$ are an edge label of $(c_1,c_2).$ With these edges, any state for which both the indegree and outdegree are zero will not be included as a vertex of $\Delta_j.$ Each $\Delta_j$ will be referred to as the {\it multi-image of $C_j,$} or simply as a {\it cycle multi-image}.
\end{definition}

We shall presently see that $C_{I}=\biguplus_{j\in I}C_j$ may be ordered into permutiple strings precisely when the corresponding multigraph union $\Delta_{I}=\biguplus_{j\in I}\Delta_j$ contains the zero state, is strongly connected, and contains an Eulerian circuit in the multigraph sense \cite{bang,farrell,silberger}. That is, every multi-edge of $\Delta_I$ must be used once when transitioning between states.

\begin{theorem}
Let $\{C_0,C_1,\ldots,C_m\}$ be the collection of cycles of $M,$ and let $\Delta_j$ be the corresponding multi-image of $C_j.$ Also, let $I$ be a multiset whose support is a subset of $\{0,1,\ldots,m\}.$ If a multigraph union of cycle multi-images $\Delta_I=\biguplus_{j\in I}\Delta_j$ contains an Eulerian circuit beginning and ending with the zero state, then the corresponding multiset union of mother-graph cycles $C_I=\biguplus_{j\in I}C_j$ may be ordered into a permutiple string.
\end{theorem}

\begin{proof}
If $\Delta_I$ contains an Eulerian circuit whose initial and final state is zero, then every labeled multi-edge is used exactly once when traversing the circuit. It follows that every edge-label input of $C_I$ is used to traverse the circuit.
\end{proof}

Since a multigraph contains an Eulerian circuit if and only if it is strongly connected, and the indegree and outdegree are equal at each vertex \cite{bang,farrell,silberger}, we may easily construct multigraphs which produce permutiple strings: form a multigraph union of cycle multi-images which 1) contains the zero state, 2) is strongly connected, and 3) the indegree is equal to the outdegree at each vertex. Any union of cycle multi-images satisfying all three of these conditions will be called {\it $L$-Eulerian}.

\begin{corollary}\label{indegree_outdegree}
Let $\{C_0,C_1,\ldots,C_m\}$ be the collection of cycles of $M,$ and let $\Delta_j$ be the corresponding multi-image of $C_j.$ Also, let $I$ be a multiset whose support is a subset of $\{0,1,\ldots,m\}.$ Then, the multiset union of mother-graph cycles $C_I=\biguplus_{j\in I}C_j$ may be ordered into a permutiple string if and only if the corresponding multigraph union of the cycle multi-images $\Delta_I=\biguplus_{j\in I}\Delta_j$ is $L$-Eulerian.
\end{corollary}

\begin{example}\label{example_5} We continue where Example \ref{example_3} leaves off.
The cycles of the $(2,4)$-mother graph and their corresponding multi-images are shown in Figure \ref{2_4_hsg}.
 \begin{center}
\begin{scriptsize}
\begin{tabular}{|c|c|l|c|}
\hline
 & {\bf Mother-Graph Cycle} & {\bf Cycle Multi-Image} &\\\hline

$C_0$ &
\begin{tikzpicture} \tikzset{edge/.style = {->,> = latex'}} \tikzset{vertex/.style = {shape=circle,draw,minimum size=1.5em}} [xscale=2, yscale=2, auto=left,every node/.style={circle,fill=blue!20}]
\node[vertex] (n0) at (4,5) {$0$};
\draw[edge](n0) to[in=30,out=-30, loop, style={min distance=7mm}] (n0);
\end{tikzpicture}
&
\begin{tikzpicture} \tikzset{edge/.style = {->,> = latex'}} \tikzset{vertex/.style = {shape=circle,draw,minimum size=1.5em}} [xscale=2, yscale=2, auto=left,every node/.style={circle,fill=blue!20}]
\node[vertex,accepting,initial] (n0) at (11,5.5) {$0$};
\node[vertex,white] (n1) at (14,5.5) {$1$};
\draw[edge](n0) to[in=150,out=90, loop, style={min distance=7mm}] node[above] {$(0,0)$} (n0);
\end{tikzpicture}
&
$\Delta_0$
\\\hline

$C_1$ &
\begin{tikzpicture} \tikzset{edge/.style = {->,> = latex'}} \tikzset{vertex/.style = {shape=circle,draw,minimum size=1.5em}} [xscale=2, yscale=2, auto=left,every node/.style={circle,fill=blue!20}]
\node[vertex] (n3) at (5.99999601961531,4.000000000007922) {$3$};
\draw[edge](n3) to[in=300,out=240, loop, style={min distance=7mm}] (n3);\end{tikzpicture}
&
\color{white}
\begin{tikzpicture} \tikzset{edge/.style = {->,> = latex'}} \tikzset{vertex/.style = {shape=circle,draw,minimum size=1.5em}} [xscale=2, yscale=2, auto=left,every node/.style={circle,fill=blue!20}]
\node[vertex,accepting,initial] (n0) at (11,5.5) {$0$};
\node[vertex,black] (n1) at (14,5.5) {$1$};
\draw[edge,black](n1) to[in=90,out=30, loop, style={min distance=7mm}] node[above] {$(3,3)$} (n1);
\end{tikzpicture}
&
$\Delta_1$
\\\hline

$C_2$ &
\begin{tikzpicture} \tikzset{edge/.style = {->,> = latex'}} \tikzset{vertex/.style = {shape=circle,draw,minimum size=1.5em}} [xscale=2, yscale=2, auto=left,every node/.style={circle,fill=blue!20}]
\node[vertex] (n1) at (6.000001326794896,5.99999999999912) {$1$};
\node[vertex] (n2) at (5.00000000000352,5.000002653589793) {$2$};
\draw[edge, bend right=10](n1) to (n2);
\draw[edge,bend right=10](n2) to (n1);
\end{tikzpicture}

&

\begin{tikzpicture} \tikzset{edge/.style = {->,> = latex'}} \tikzset{vertex/.style = {shape=circle,draw,minimum size=1.5em}} [xscale=2, yscale=2, auto=left,every node/.style={circle,fill=blue!20}]
\node[vertex,accepting,initial] (n0) at (11,5.5) {$0$};
\node[vertex] (n1) at (14,5.5) {$1$};
\draw[edge](n0) to[in=270,out=210, loop, style={min distance=7mm}] node[below] {$(2,1)$} (n0);
\draw[edge](n1) to[in=-30,out=-90, loop, style={min distance=7mm}] node[below] {$(1,2)$} (n1);
\end{tikzpicture}
&
$\Delta_2$
\\\hline

$C_3$ &
\begin{tikzpicture} \tikzset{edge/.style = {->,> = latex'}} \tikzset{vertex/.style = {shape=circle,draw,minimum size=1.5em}} [xscale=2, yscale=2, auto=left,every node/.style={circle,fill=blue!20}]
\node[vertex] (n0) at (7,5) {$0$};
\node[vertex] (n1) at (6.000001326794896,5.99999999999912) {$1$};
\node[vertex] (n2) at (5.00000000000352,5.000002653589793) {$2$};
\draw[edge](n0) to (n2);
\draw[edge](n1) to (n0);
\draw[edge](n2) to (n1);
\end{tikzpicture}

&

\begin{tikzpicture} \tikzset{edge/.style = {->,> = latex'}} \tikzset{vertex/.style = {shape=circle,draw,minimum size=1.5em}} [xscale=2, yscale=2, auto=left,every node/.style={circle,fill=blue!20}]
\node[vertex,accepting,initial] (n0) at (11,5.5) {$0$};
\node[vertex] (n1) at (14,5.5) {$1$};
\draw[edge, bend right=10] (n0) edge node[below] {$(0,2)$} (n1);
\draw[edge, bend right=10] (n1) edge node[above] {$(1,0)$} (n0);
\draw[edge](n0) to[in=270,out=210, loop, style={min distance=7mm}] node[below] {$(2,1)$} (n0);
\end{tikzpicture}
&
$\Delta_3$
\\\hline

$C_4$ &
\begin{tikzpicture} \tikzset{edge/.style = {->,> = latex'}} \tikzset{vertex/.style = {shape=circle,draw,minimum size=1.5em}} [xscale=2, yscale=2, auto=left,every node/.style={circle,fill=blue!20}]
\node[vertex] (n1) at (6.000001326794896,5.99999999999912) {$1$};
\node[vertex] (n2) at (5.00000000000352,5.000002653589793) {$2$};
\node[vertex] (n3) at (5.99999203923062,4.0000000000158433) {$3$};
\draw[edge](n1) to (n2);
\draw[edge](n2) to (n3);
\draw[edge](n3) to (n1);
\end{tikzpicture}

&

\begin{tikzpicture} \tikzset{edge/.style = {->,> = latex'}} \tikzset{vertex/.style = {shape=circle,draw,minimum size=1.5em}} [xscale=2, yscale=2, auto=left,every node/.style={circle,fill=blue!20}]

\node[vertex,accepting,initial] (n0) at (11,5.5) {$0$};
\node[vertex] (n1) at (14,5.5) {$1$};
\draw[edge, bend right=10] (n0) edge node[below] {$(2,3)$} (n1);
\draw[edge, bend right=10] (n1) edge node[above] {$(3,1)$} (n0);
\draw[edge](n1) to[in=-30,out=-90, loop, style={min distance=7mm}] node[below] {$(1,2)$} (n1);
\end{tikzpicture}
&
$\Delta_4$
\\\hline

$C_5$ &
\begin{tikzpicture} \tikzset{edge/.style = {->,> = latex'}} \tikzset{vertex/.style = {shape=circle,draw,minimum size=1.5em}} [xscale=2, yscale=2, auto=left,every node/.style={circle,fill=blue!20}]
\node[vertex] (n0) at (7,5) {$0$};
\node[vertex] (n1) at (6.000001326794896,5.99999999999912) {$1$};
\node[vertex] (n2) at (5.00000000000352,5.000002653589793) {$2$};
\node[vertex] (n3) at (5.99999203923062,4.0000000000158433) {$3$};
\draw[edge](n0) to (n2);
\draw[edge](n1) to (n0);
\draw[edge](n2) to (n3);
\draw[edge](n3) to (n1);
\end{tikzpicture}

&

\begin{tikzpicture} \tikzset{edge/.style = {->,> = latex'}} \tikzset{vertex/.style = {shape=circle,draw,minimum size=1.5em}} [xscale=2, yscale=2, auto=left,every node/.style={circle,fill=blue!20}]
\node[vertex,accepting,initial] (n0) at (11,5.5) {$0$};
\node[vertex] (n1) at (14,5.5) {$1$};

\draw[edge, bend right=50] (n1) edge node[above] {$(3,1)$} (n0);
\draw[edge, bend right=10] (n1) edge node[above] {$(1,0)$} (n0);

\draw[edge, bend right=10] (n0) edge node[below] {$(2,3)$} (n1);
\draw[edge, bend right=50] (n0) edge node[below] {$(0,2)$} (n1);

\end{tikzpicture}
&
$\Delta_5$
\\\hline
\end{tabular}
\captionof{table}{Cycles of the $(2,4)$-mother graph and corresponding cycle multi-images.}
\label{2_4_ci}
\end{scriptsize}
\end{center}

Since $\Delta_0,$ $\Delta_3,$ $\Delta_4,$ and $\Delta_5$ are all $L$-Eulerian, we may form permutiple strings from these graphs individually. Moreover, since the indegrees and outdegrees are equal for all vertices for each $\Delta_j,$ we may say that any multigraph union of cycle multi-images which involves $\Delta_3,$ $\Delta_4,$ and $\Delta_5$ will enable us to form permutiple strings. Permutiples which may be formed from  multi-images of individual cycles mentioned above are given in Table \ref{2_4_indiv}.
\begin{center}
\begin{footnotesize}
 \begin{tabular}{|c|l|l|}
 \hline
 Cycle Multi-Image & Permutiple String & Permutiple Example\\\hline
$\Delta_0$ & $(0,0)$ &$(0)_4=2\cdot(0)_4$\\\hline
$\Delta_3$ & $(2,1)(0,2)(1,0)$ & $(1,0,2)_4=2\cdot(0,2,1)_4$\\
& $(0,2)(1,0)(2,1)$ &$(2,1,0)_4=2\cdot(1,0,2)_4$\\\hline
$\Delta_4$ & $(2,3)(1,2)(3,1)$ & $(3,1,2)_4=2\cdot(1,2,3)_4$\\\hline
$\Delta_5$ & $(0,2)(1,0)(2,3)(3,1)$& $(3,2,1,0)_4=2\cdot(1,3,0,2)_4$\\
 & $(0,2)(3,1)(2,3)(1,0)$& $(1,2,3,0)_4=2\cdot(0,3,1,2)_4$\\
 & $(2,3)(1,0)(0,2)(3,1)$& $(3,0,1,2)_4=2\cdot(1,2,0,3)_4$\\
 & $(2,3)(3,1)(0,2)(1,0)$& $(1,0,3,2)_4=2\cdot(0,2,1,3)_4$\\\hline
\end{tabular}
\end{footnotesize}
\captionof{table}{Examples of $(2,4)$-permutiples formed from $\Delta_0,$ $\Delta_3,$ $\Delta_4,$ and $\Delta_5.$}
\label{2_4_indiv}
\end{center}

A less trivial instance is the multiset cycle union $C_2 \uplus C_3,$ which gives a collection of inputs with which we may form permutiple strings since the corresponding multi-image union $\Delta_2 \uplus \Delta_3$ is $L$-Eulerian. This is verified by examining $\Delta_2 \uplus \Delta_3$ in Figure \ref{2_4_mi_union_1}. We may apply Corollary \ref{indegree_outdegree} to deduce that we may order $C_2 \uplus C_3$ into permutiple strings using $\Delta_2 \uplus \Delta_3.$
\begin{center}
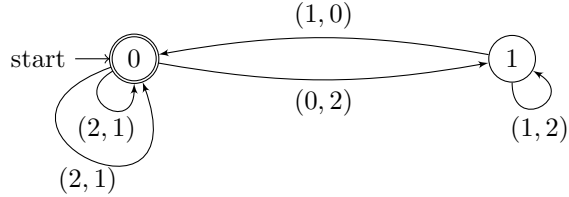

\begin{tikzpicture} \tikzset{edge/.style = {->,> = latex'}} \tikzset{vertex/.style = {shape=circle,draw,minimum size=1.5em}} [xscale=2, yscale=2, auto=left,every node/.style={circle,fill=blue!20}]
\node[vertex,accepting,initial] (n0) at (0,0) {$0$};
\node[vertex] (n1) at (5,0) {$1$};
\draw[edge, bend right=10] (n0) edge node[below] {$(0,2)$} (n1);
\draw[edge, bend right=10] (n1) edge node[above] {$(1,0)$} (n0);
\draw[edge](n0) to[in=270,out=210, loop, style={min distance=7mm}] node[below] {$(2,1)$} (n0);
\draw[edge](n0) to[in=290,out=200, loop, style={min distance=23mm}] node[below] {$(2,1)$} (n0);
\draw[edge](n1) to[in=-30,out=-90, loop, style={min distance=7mm}] node[below] {$(1,2)$} (n1);
\end{tikzpicture}
\captionof{figure}{The multi-image union $\Delta_I=\Delta_2 \uplus \Delta_3$ corresponding to the multiset union $C_I=C_2 \uplus C_3$ of mother-graph cycles.}
\label{2_4_mi_union_1}
\end{center}
All numerically distinct, 5-digit examples obtained from Eulerian circuits on $\Delta_2 \uplus \Delta_3,$ beginning and ending with the zero state, are included in Table \ref{2_4_union_1}.
\begin{center}
\begin{footnotesize}
 \begin{tabular}{|c|l|l|}
 \hline
 Multi-Image Union & Permutiple String & Example\\\hline
$\Delta_2 \uplus \Delta_3$ & $(2,1)(0,2)(1,2)(1,0)(2,1)$ &$(2,1,1,0,2)_4=2\cdot(1,0,2,2,1)_4$\\
& $(2,1)(2,1)(0,2)(1,2)(1,0)$ &$(1,1,0,2,2)_4=2\cdot(0,2,2,1,1)_4$\\
& $(0,2)(1,2)(1,0)(2,1)(2,1)$ &$(2,2,1,1,0)_4=2\cdot(1,1,0,2,2)_4$\\\hline
\end{tabular}
\end{footnotesize}
\captionof{table}{Numerically distinct, 5-digit $(2,4)$-permutiples formed from the multigraph union $\Delta_2 \uplus \Delta_3.$}
\label{2_4_union_1}
\end{center}
As mentioned earlier, as well as in \cite{holt_5}, the multiset context allows for cycles to appear more than once. We take the multiset union
\[
C_3 \uplus C_3=\{(0,2),(2,1),(1,0),(0,2),(2,1),(1,0)\}
\]
as an example which gives us several six-digit permutiples. We first examine its corresponding multi-image union $\Delta_3 \uplus \Delta_3$ shown in Figure \ref{2_4_mi_union_2}.

\begin{center}
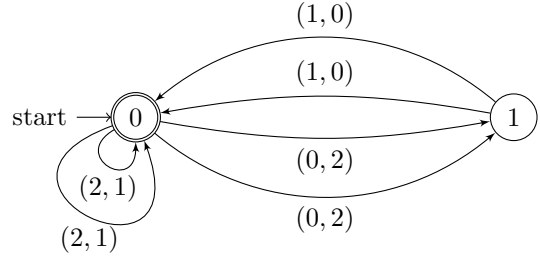

\begin{tikzpicture} \tikzset{edge/.style = {->,> = latex'}} \tikzset{vertex/.style = {shape=circle,draw,minimum size=1.5em}} [xscale=2, yscale=2, auto=left,every node/.style={circle,fill=blue!20}]
\node[vertex,accepting,initial] (n0) at (0,0) {$0$};
\node[vertex] (n1) at (5,0) {$1$};
\draw[edge, bend right=10] (n0) edge node[below] {$(0,2)$} (n1);

\draw[edge, bend right=40] (n0) edge node[below] {$(0,2)$} (n1);

\draw[edge, bend right=10] (n1) edge node[above] {$(1,0)$} (n0);
\draw[edge, bend right=40] (n1) edge node[above] {$(1,0)$} (n0);

\draw[edge](n0) to[in=270,out=210, loop, style={min distance=7mm}] node[below] {$(2,1)$} (n0);
\draw[edge](n0) to[in=290,out=200, loop, style={min distance=23mm}] node[below] {$(2,1)$} (n0);
\end{tikzpicture}
\captionof{figure}{The multi-image union $\Delta_I=\Delta_3 \uplus \Delta_3$ corresponding to the multiset union $C_I=C_3 \uplus C_3$ of mother-graph cycles.}
\label{2_4_mi_union_2}
\end{center}
Again, we see that the multi-image union $\Delta_3 \uplus \Delta_3$ is $L$-Eulerian. The numerically distinct permutiple strings and permutiple examples formed from Eulerian circuits beginning and ending with the zero state are displayed in Table \ref{2_4_union_2}.

\begin{center}
\begin{footnotesize}
 \begin{tabular}{|c|l|l|}
 \hline
Multi-Image Union & Permutiple String & Example\\\hline
$\Delta_3 \uplus \Delta_3$ & $(2,1)(0,2)(1,0)(2,1)(0,2)(1,0)$ &$(1,0,2,1,0,2)_4=2\cdot(0,2,1,0,2,1)_4$\\
& $(0,2)(1,0)(2,1)(0,2)(1,0)(2,1)$ & $(2,1,0,2,1,0)_4=2\cdot(1,0,2,1,0,2)_4$\\
& $ (2,1)(0,2)(1,0)(0,2)(1,0)(2,1)$ & $(2,1,0,1,0,2)_4=2\cdot(1,0,2,0,2,1)_4$\\
& $ (2,1)(2,1)(0,2)(1,0)(0,2)(1,0)$ & $(1,0,1,0,2,2)_4=2\cdot(0,2,0,2,1,1)_4$\\
& $(0,2)(1,0)(0,2)(1,0)(2,1)(2,1)$ & $(2,2,1,0,1,0)_4=2\cdot(1,1,0,2,0,2)_4$\\
& $(0,2)(1,0)(2,1)(2,1)(0,2)(1,0)$ & $(1,0,2,2,1,0)_4=2\cdot(0,2,1,1,0,2)_4$\\
\hline
\end{tabular}
\end{footnotesize}
\captionof{table}{Numerically distinct, 6-digit $(2,4)$-permutiples formed from the multigraph union $\Delta_3 \uplus \Delta_3.$}
\label{2_4_union_2}
\end{center}

As already mentioned, since all of the cycle multi-images have equal indegree and outdegree at each vertex, any multigraph union of them will also share this property. It follows that all nontrivial $(2,4)$-permutiple strings can be formed from multigraph unions of cycle multi-images which contain at least one copy of $\Delta_3,$ $\Delta_4,$ or $\Delta_5.$
\end{example}

\begin{example}\label{example_6} We now continue with Example \ref{example_4}. The cycles of the $(3,4)$-mother graph and their corresponding cycle multi-images are seen in Table \ref{3_4_mi}.
\begin{center}
\begin{scriptsize}
\begin{tabular}{|c|c|l|c|}
\hline
& {\bf Mother-Graph Cycle} & {\bf Cycle Multi-Image}&\\\hline

$C_0$ &
\begin{tikzpicture} \tikzset{edge/.style = {->,> = latex'}} \tikzset{vertex/.style = {shape=circle,draw,minimum size=1.5em}} [xscale=2, yscale=2, auto=left,every node/.style={circle,fill=blue!20}]
\node[vertex] (n0) at (7,5) {$0$};
\draw[edge](n0) to[in=30,out=-30, loop, style={min distance=5mm}] (n0);
\end{tikzpicture}

&

\begin{tikzpicture} \tikzset{edge/.style = {->,> = latex'}} \tikzset{vertex/.style = {shape=circle,draw,minimum size=1.5em}} [xscale=2, yscale=2, auto=left,every node/.style={circle,fill=blue!20}]
\node[vertex,initial,accepting] (n0) at (0,5) {$0$};
\node[vertex,white] (n1) at (2,5) {$1$};
\node[vertex,white] (n2) at (4,5) {$2$};
\draw[edge](n0) to[in=150,out=90, loop, style={min distance=5mm}] node[pos=0.25,right] {$(0,0)$} (n0);
\end{tikzpicture}
&
$\Delta_0$
\\\hline

$C_1$ &
\begin{tikzpicture} \tikzset{edge/.style = {->,> = latex'}} \tikzset{vertex/.style = {shape=circle,draw,minimum size=1.5em}} [xscale=2, yscale=2, auto=left,every node/.style={circle,fill=blue!20}]
\node[vertex] (n1) at (6.000001326794896,5.99999999999912) {$1$};
\draw[edge](n1) to[in=120,out=60, loop, style={min distance=5mm}] (n1);
\end{tikzpicture}

&

\begin{tikzpicture} \tikzset{edge/.style = {->,> = latex'}} \tikzset{vertex/.style = {shape=circle,draw,minimum size=1.5em}} [xscale=2, yscale=2, auto=left,every node/.style={circle,fill=blue!20}]
\color{white}\node[vertex,initial,accepting] (n0) at (0,5) {$0$};
\node[vertex,black] (n1) at (2,5) {$1$};
\node[vertex,black] (n2) at (4,5) {$2$};
\draw[edge, black, bend left=10](n2) edge node[below] {$(1,1)$} (n1);
\end{tikzpicture}
&
$\Delta_1$
\\\hline

$C_2$ &
\begin{tikzpicture} \tikzset{edge/.style = {->,> = latex'}} \tikzset{vertex/.style = {shape=circle,draw,minimum size=1.5em}} [xscale=2, yscale=2, auto=left,every node/.style={circle,fill=blue!20}]
\node[vertex] (n2) at (5.00000000000352,5.000002653589793) {$2$};
\draw[edge](n2) to[in=200,out=140, loop, style={min distance=5mm}] (n2);
\end{tikzpicture}

&

\begin{tikzpicture} \tikzset{edge/.style = {->,> = latex'}} \tikzset{vertex/.style = {shape=circle,draw,minimum size=1.5em}} [xscale=2, yscale=2, auto=left,every node/.style={circle,fill=blue!20}]
\node[vertex,initial,accepting] (n0) at (0,5) {$0$};
\node[vertex] (n1) at (2,5) {$1$};
\node[vertex,white] (n2) at (4,5) {$2$};
\draw[edge, bend right=10](n0) edge node[below] {$(2,2)$} (n1);
\end{tikzpicture}
&
$\Delta_2$
\\\hline

$C_3$ &
\begin{tikzpicture} \tikzset{edge/.style = {->,> = latex'}} \tikzset{vertex/.style = {shape=circle,draw,minimum size=1.5em}} [xscale=2, yscale=2, auto=left,every node/.style={circle,fill=blue!20}]
 \node[vertex] (n3) at (5.99999203923062,4.0000000000158433) {$3$};
\draw[edge](n3) to[in=300,out=240, loop, style={min distance=5mm}] (n3);
 \end{tikzpicture}

&
\color{white}
\begin{tikzpicture} \tikzset{edge/.style = {->,> = latex'}} \tikzset{vertex/.style = {shape=circle,draw,minimum size=1.5em}} [xscale=2, yscale=2, auto=left,every node/.style={circle,fill=blue!20}]
\node[vertex,initial,accepting] (n0) at (0,5) {$0$};
\node[vertex] (n1) at (2,5) {$1$};
\node[vertex,black] (n2) at (4,5) {$2$};
 \draw[edge,black](n2) to[in=-30,out=-90, loop, style={min distance=5mm}] node[pos=0.75,right] {$(3,3)$} (n2);
\end{tikzpicture}
&
$\Delta_3$
\\\hline

$C_4$ &
\begin{tikzpicture} \tikzset{edge/.style = {->,> = latex'}} \tikzset{vertex/.style = {shape=circle,draw,minimum size=1.5em}} [xscale=2, yscale=2, auto=left,every node/.style={circle,fill=blue!20}]
\node[vertex] (n0) at (7,5) {$0$};
\node[vertex] (n1) at (6.000001326794896,5.99999999999912) {$1$};
\draw[edge, bend right=10](n1) to (n0);
\draw[edge, bend right=10](n0) to (n1);
\end{tikzpicture}

&

\begin{tikzpicture} \tikzset{edge/.style = {->,> = latex'}} \tikzset{vertex/.style = {shape=circle,draw,minimum size=1.5em}} [xscale=2, yscale=2, auto=left,every node/.style={circle,fill=blue!20}]
\node[vertex,initial,accepting] (n0) at (0,5) {$0$};
\node[vertex] (n1) at (2,5) {$1$};
\node[vertex,white] (n2) at (4,5) {$2$};
\draw[edge](n1) to[in=120,out=60, loop, style={min distance=5mm}] node[above] {$(0,1)$} (n1);
\draw[edge, bend right=10](n1) edge node[above] {$(1,0)$} (n0);
\end{tikzpicture}
&
$\Delta_4$
\\\hline

$C_5$ &
\begin{tikzpicture} \tikzset{edge/.style = {->,> = latex'}} \tikzset{vertex/.style = {shape=circle,draw,minimum size=1.5em}} [xscale=2, yscale=2, auto=left,every node/.style={circle,fill=blue!20}]
\node[vertex] (n0) at (7,5) {$0$};
\node[vertex] (n2) at (5.00000000000352,5.000002653589793) {$2$};
\draw[edge, bend right=10](n0) to (n2);
\draw[edge, bend right=10](n2) to (n0);
\end{tikzpicture}

&

\begin{tikzpicture} \tikzset{edge/.style = {->,> = latex'}} \tikzset{vertex/.style = {shape=circle,draw,minimum size=1.5em}} [xscale=2, yscale=2, auto=left,every node/.style={circle,fill=blue!20}]
\node[vertex,initial,accepting] (n0) at (0,5) {$0$};
\node[vertex,white] (n1) at (2,5) {$1$};
\node[vertex] (n2) at (4,5) {$2$};
\draw[edge, bend right=10](n2) edge node[above] {$(2,0)$} (n0);
\draw[edge](n2) to[in=90,out=30, loop, style={min distance=5mm}] node[pos=0.25,right] {$(0,2)$} (n2);
\end{tikzpicture}
&
$\Delta_5$
\\\hline

$C_6$ &
\begin{tikzpicture} \tikzset{edge/.style = {->,> = latex'}} \tikzset{vertex/.style = {shape=circle,draw,minimum size=1.5em}} [xscale=2, yscale=2, auto=left,every node/.style={circle,fill=blue!20}]
\node[vertex] (n1) at (0,0) {$1$};
\node[vertex] (n3) at (0,1.5) {$3$};
\draw[edge, bend right=10](n3) to (n1);
\draw[edge, bend right=10](n1) to (n3);
\end{tikzpicture}

&

\begin{tikzpicture} \tikzset{edge/.style = {->,> = latex'}} \tikzset{vertex/.style = {shape=circle,draw,minimum size=1.5em}} [xscale=2, yscale=2, auto=left,every node/.style={circle,fill=blue!20}]
\node[vertex,initial,accepting] (n0) at (0,5) {$0$};
\node[vertex,white] (n1) at (2,5) {$1$};
\node[vertex] (n2) at (4,5) {$2$};
\draw[edge, bend right=10](n0) edge node[below] {$(1,3)$} (n2);
\draw[edge](n0) to[in=270,out=210, loop, style={min distance=5mm}] node[below] {$(3,1)$} (n0);
\end{tikzpicture}
&
$\Delta_6$
\\\hline

$C_7$ &
\begin{tikzpicture} \tikzset{edge/.style = {->,> = latex'}} \tikzset{vertex/.style = {shape=circle,draw,minimum size=1.5em}} [xscale=2, yscale=2, auto=left,every node/.style={circle,fill=blue!20}]
\node[vertex] (n2) at (5.00000000000352,5.000002653589793) {$2$};
\node[vertex] (n3) at (5.99999203923062,4.0000000000158433) {$3$};
\draw[edge, bend right=10](n2) to (n3);
\draw[edge, bend right=10](n3) to (n2);
\end{tikzpicture}

&
\color{white}
\begin{tikzpicture} \tikzset{edge/.style = {->,> = latex'}} \tikzset{vertex/.style = {shape=circle,draw,minimum size=1.5em}} [xscale=2, yscale=2, auto=left,every node/.style={circle,fill=blue!20}]
\node[vertex,initial,accepting] (n0) at (0,5) {$0$};
\node[vertex,black] (n1) at (2,5) {$1$};
\node[vertex,black] (n2) at (4,5) {$2$};
\draw[edge, black, bend left=10](n1) edge node[above] {$(2,3)$} (n2);
\draw[edge, black](n1) to[in=300,out=240, loop, style={min distance=5mm}] node[pos=0.75,right] {$(3,2)$} (n1);
\end{tikzpicture}
&
$\Delta_7$
\\\hline

$C_8$ &
\begin{tikzpicture} \tikzset{edge/.style = {->,> = latex'}} \tikzset{vertex/.style = {shape=circle,draw,minimum size=1.5em}} [xscale=2, yscale=2, auto=left,every node/.style={circle,fill=blue!20}]
\node[vertex] (n0) at (7,5) {$0$};
\node[vertex] (n1) at (6.000001326794896,5.99999999999912) {$1$};
\node[vertex] (n2) at (5.00000000000352,5.000002653589793) {$2$};
\node[vertex] (n3) at (5.99999203923062,4.0000000000158433) {$3$};
\draw[edge](n2) to (n0);
\draw[edge](n0) to (n1);
\draw[edge](n3) to (n2);
\draw[edge](n1) to (n3);
\end{tikzpicture}

&

\begin{tikzpicture} \tikzset{edge/.style = {->,> = latex'}} \tikzset{vertex/.style = {shape=circle,draw,minimum size=1.5em}} [xscale=2, yscale=2, auto=left,every node/.style={circle,fill=blue!20}]
\node[vertex,initial,accepting] (n0) at (0,5) {$0$};
\node[vertex] (n1) at (2,5) {$1$};
\node[vertex] (n2) at (4,5) {$2$};
\draw[edge, bend right=40](n2) edge node[above] {$(2,0)$} (n0);
\draw[edge](n1) to[in=120,out=60, loop, style={min distance=4mm}] node[pos=0.25,right] {$(0,1)$} (n1);
\draw[edge](n1) to[in=300,out=240, loop, style={min distance=4mm}] node[pos=0.25,left] {$(3,2)$} (n1);
\draw[edge, bend right=40](n0) edge node[below] {$(1,3)$} (n2);
\end{tikzpicture}
&
$\Delta_8$
\\\hline

$C_9$ &
\begin{tikzpicture} \tikzset{edge/.style = {->,> = latex'}} \tikzset{vertex/.style = {shape=circle,draw,minimum size=1.5em}} [xscale=2, yscale=2, auto=left,every node/.style={circle,fill=blue!20}]
\node[vertex] (n0) at (7,5) {$0$};
\node[vertex] (n1) at (6.000001326794896,5.99999999999912) {$1$};
\node[vertex] (n2) at (5.00000000000352,5.000002653589793) {$2$};
\node[vertex] (n3) at (5.99999203923062,4.0000000000158433) {$3$};
\draw[edge](n0) to (n2);
\draw[edge](n1) to (n0);
\draw[edge](n2) to (n3);
\draw[edge](n3) to (n1);
\end{tikzpicture}

&

\begin{tikzpicture} \tikzset{edge/.style = {->,> = latex'}} \tikzset{vertex/.style = {shape=circle,draw,minimum size=1.5em}} [xscale=2, yscale=2, auto=left,every node/.style={circle,fill=blue!20}]
\node[vertex,initial,accepting] (n0) at (0,5) {$0$};
\node[vertex] (n1) at (2,5) {$1$};
\node[vertex] (n2) at (4,5) {$2$};
\draw[edge, bend right=10](n1) edge node[above] {$(1,0)$} (n0);
\draw[edge, bend left=10](n1) edge node[above] {$(2,3)$} (n2);
\draw[edge](n0) to[in=270,out=210, loop, style={min distance=5mm}] node[below] {$(3,1)$} (n0);
\draw[edge](n2) to[in=90,out=30, loop, style={min distance=5mm}] node[above] {$(0,2)$} (n2);
\end{tikzpicture}
&
$\Delta_9$
\\\hline

\end{tabular}
\captionof{table}{Cycles of the $(3,4)$-mother graph and corresponding cycle multi-images.}
\label{3_4_mi}
\end{scriptsize}
\end{center}

We may now consider multiset unions of mother-graph cycles whose corresponding cycle multi-image union is $L$-Eulerian. In this way, Table \ref{3_4_mi} describes how to form any $(3,4)$-permutiple string with a suitable number of digits. For instance, the multiset union of mother-graph cycles $C_I=C_1 \uplus C_2 \uplus C_5 \uplus C_6 \uplus C_9$ corresponds to the multi-image union $\Delta_I=\Delta_1 \uplus \Delta_2 \uplus \Delta_5 \uplus \Delta_6 \uplus \Delta_9,$ which is $L$-Eulerian. This multigraph is shown in Figure \ref{3_4_mi_union_1}.

\begin{center}
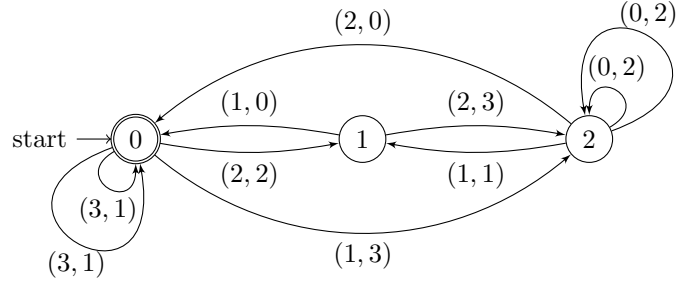

 \begin{tikzpicture} \tikzset{edge/.style = {->,> = latex'}} \tikzset{vertex/.style = {shape=circle,draw,minimum size=1.5em}} [xscale=2, yscale=2, auto=left,every node/.style={circle,fill=blue!20}]
\node[vertex,initial,accepting] (n0) at (0,0) {$0$};
\node[vertex] (n1) at (3,0) {$1$};
\node[vertex] (n2) at (6,0) {$2$};
\draw[edge, bend right=10](n1) edge node[above] {$(1,0)$} (n0);
\draw[edge, bend left=10](n1) edge node[above] {$(2,3)$} (n2);
\draw[edge](n0) to[in=270,out=210, loop, style={min distance=7mm}] node[below] {$(3,1)$} (n0);
\draw[edge](n2) to[in=90,out=30, loop, style={min distance=7mm}] node[above] {$(0,2)$} (n2);

\draw[edge, bend right=40](n0) edge node[below] {$(1,3)$} (n2);
\draw[edge](n0) to[in=280,out=200, loop, style={min distance=23mm}] node[below] {$(3,1)$} (n0);
\draw[edge, bend right=40](n2) edge node[above] {$(2,0)$} (n0);
\draw[edge](n2) to[in=100,out=20, loop, style={min distance=23mm}] node[above] {$(0,2)$} (n2);
\draw[edge, bend right=10](n0) edge node[below] {$(2,2)$} (n1);
\draw[edge, black, bend left=10](n2) edge node[below] {$(1,1)$} (n1);
\end{tikzpicture}
\captionof{figure}{The multi-image union $\Delta_I=\Delta_1 \uplus \Delta_2 \uplus \Delta_5 \uplus \Delta_6 \uplus \Delta_9$ corresponding to the multiset union $C_I=C_1 \uplus C_2 \uplus C_5 \uplus C_6 \uplus C_9$ of mother-graph cycles.}
\label{3_4_mi_union_1}
\end{center}
We may use Figure \ref{3_4_mi_union_1} to form permutiple strings by traversing Eulerian circuits on $\Delta_I$ which begin and end with the zero state, such as
\[
(1,3)(0,2)(1,1)(1,0)(3,1)(2,2)(2,3)(0,2)(2,0)(3,1),
\]
which yields the permutiple
$(3,2,0,2,2,3,1,1,0,1)_4=3\cdot(1,0,2,3,2,1,0,1,2,3)_4.$
\end{example}

For our final example, we apply the above techniques to the $(4,10)$-permutiple class $C$ whose graph $G_C$ is featured in Figures \ref{conj_class_graph} and \ref{4_10_mg}.

\begin{example}\label{example_7}
 We begin by examining the multi-images of the cycles of $G_C.$

 \begin{center}
\begin{scriptsize}
\begin{tabular}{|c|c|l|c|}
\hline
& {\bf Cycle of $G_C$} & {\bf Cycle  Multi-Image}&\\\hline

$C_0$ &
\begin{tikzpicture} \tikzset{edge/.style = {->,> = latex'}} \tikzset{vertex/.style = {shape=circle,draw,minimum size=1.5em}} [xscale=2, yscale=2, auto=left,every node/.style={circle,fill=blue!20}]
\node[vertex] (n1) at (0,0) {$9$};
\draw[edge](n1) to[in=0,out=-80, loop, style={min distance=7mm}] (n1);
\end{tikzpicture}

&

\begin{tikzpicture}
\tikzset{edge/.style = {->,> = latex'}}
\tikzset{vertex/.style = {shape=circle,draw,minimum size=1.5em}}
[xscale=2, yscale=2, auto=left,every node/.style={circle,fill=blue!20}]
\color{white}\node[vertex,accepting,initial,white] (n0) at (0,5) {$0$};
\color{black}\node[vertex] (n1) at (3,5) {$3$};

\color{black}\draw[edge](n1) to[in=130,out=50, loop, style={min distance=7mm}] node[pos=0.25,right] {$(9,9)$} (n1);
\end{tikzpicture}
&
$\Delta_0$
\\\hline

$C_1$ &
\begin{tikzpicture}
\tikzset{edge/.style = {->,> = latex'}}
\tikzset{vertex/.style = {shape=circle,draw,minimum size=1.5em}}
[xscale=3, yscale=3, auto=left,every node/.style={circle,fill=blue!20}]
\node[vertex] (n2) at (0,1.3) {$2$};
\node[vertex] (n8) at (0,0) {$8$};
\draw[edge, bend right=10] (n8) to (n2);
\draw[edge, bend right=10] (n2) to (n8);
\end{tikzpicture}

&

\begin{tikzpicture}
\tikzset{edge/.style = {->,> = latex'}}
\tikzset{vertex/.style = {shape=circle,draw,minimum size=1.5em}}
[xscale=2, yscale=2, auto=left,every node/.style={circle,fill=blue!20}]
\node[vertex,accepting,initial] (n0) at (0,5) {$0$};
\node[vertex] (n1) at (3,5) {$3$};
\draw[edge,bend right=10] (n0) edge node[below] {$(2,8)$} (n1);
\draw[edge](n0) to[in=150,out=90, loop, style={min distance=10mm}] node[pos=0.75,left] {$(8,2)$} (n0);
\end{tikzpicture}
&
$\Delta_1$
\\\hline

$C_2$ &
\begin{tikzpicture}
\tikzset{edge/.style = {->,> = latex'}}
\tikzset{vertex/.style = {shape=circle,draw,minimum size=1.5em}}
[xscale=3, yscale=3, auto=left,every node/.style={circle,fill=blue!20}]
\node[vertex] (n1) at (1,1.1) {$1$};
\node[vertex] (n7) at (0,0) {$7$};
\draw[edge, bend right=10] (n7) to (n1);
\draw[edge, bend right=10] (n1) to (n7);
\end{tikzpicture}

&

\begin{tikzpicture}
\tikzset{edge/.style = {->,> = latex'}}
\tikzset{vertex/.style = {shape=circle,draw,minimum size=1.5em}}
[xscale=2, yscale=2, auto=left,every node/.style={circle,fill=blue!20}]
\node[vertex,accepting,initial] (n0) at (0,5) {$0$};
\node[vertex] (n1) at (3,5) {$3$};
\draw[edge, bend right=10] (n1) edge node[above] {$(7,1)$} (n0);
\draw[edge](n1) to[in=90,out=30, loop, style={min distance=10mm}] node[pos=0.25,right] {$(1,7)$} (n1);
\end{tikzpicture}
&
$\Delta_2$
\\\hline
\end{tabular}
\end{scriptsize}
\captionof{table}{Cycles of $G_C$ and their corresponding cycle multi-images.}\label{4_10_mi}
\end{center}
As shown in Example \ref{example_2}, the multiset union
\[
C_0 \uplus C_1\uplus C_1 \uplus C_2 \uplus C_2
=\{(9,9),(8,2),(8,2),(2,8),(2,8),(7,1),(7,1),(1,7),(1,7)\}
\]
may be ordered into permutiple strings. We verify this by examining the multi-image union $\Delta_0 \uplus \Delta_1 \uplus \Delta_1 \uplus \Delta_2 \uplus \Delta_2$ corresponding to the above multiset union of cycles, shown in Figure \ref{4_10_mi_union_1}.

\begin{center}
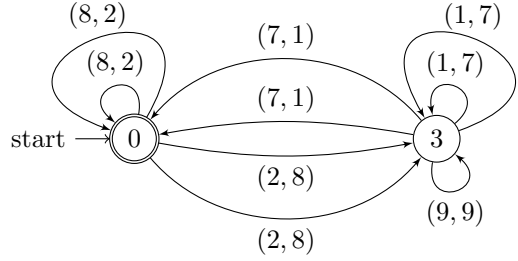

\begin{tikzpicture}
\tikzset{edge/.style = {->,> = latex'}}
\tikzset{vertex/.style = {shape=circle,draw,minimum size=1.5em}}
[xscale=2, yscale=2, auto=left,every node/.style={circle,fill=blue!20}]
\node[vertex,accepting,initial] (n0) at (0,0) {$0$};
\node[vertex] (n1) at (4,0) {$3$};
\draw[edge,bend right=10] (n0) edge node[below] {$(2,8)$} (n1);

\draw[edge,bend right=50] (n0) edge node[below] {$(2,8)$} (n1);

\draw[edge, bend right=10] (n1) edge node[above] {$(7,1)$} (n0);

\draw[edge, bend right=50] (n1) edge node[above] {$(7,1)$} (n0);

\draw[edge](n0) to[in=140,out=80, loop, style={min distance=7mm}] node[above] {$(8,2)$} (n0);
\draw[edge](n0) to[in=160,out=60, loop, style={min distance=25mm}] node[above] {$(8,2)$} (n0);

\draw[edge](n1) to[in=100,out=40, loop, style={min distance=7mm}] node[above] {$(1,7)$} (n1);
\draw[edge](n1) to[in=120,out=20, loop, style={min distance=25mm}] node[above] {$(1,7)$} (n1);

\draw[edge](n1) to[in=-40,out=-100, loop, style={min distance=7mm}] node[below] {$(9,9)$} (n1);

\end{tikzpicture}
\captionof{figure}{The multi-image union $\Delta_I=\Delta_0 \uplus \Delta_1 \uplus \Delta_1 \uplus \Delta_2 \uplus \Delta_2$ corresponding to the multiset union $C_I=C_0 \uplus C_1\uplus C_1 \uplus C_2 \uplus C_2$ of mother-graph cycles.}
\label{4_10_mi_union_1}
\end{center}
Since the multi-image union is $L$-Eulerian, we see, by Corollary \ref{indegree_outdegree}, that we may order the multiset union $C_0 \uplus C_1\uplus C_1 \uplus C_2 \uplus C_2$ into permutiple strings by traversing Eulerian circuits beginning and ending with the zero state on $\Delta_0 \uplus \Delta_1 \uplus \Delta_1 \uplus \Delta_2 \uplus \Delta_2.$ Two examples of these can be found in Example \ref{example_2}.

On the other hand, as claimed both in Example \ref{example_2} and in \cite{holt_6}, the multiset union $C_1\uplus C_1 \uplus C_2$ cannot be ordered into permutiple strings. Again, we examine the corresponding multigraph union of multi-images $\Delta_1 \uplus \Delta_1 \uplus \Delta_2.$

\begin{center}
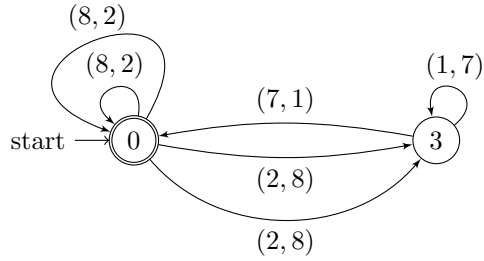

\begin{tikzpicture}
\tikzset{edge/.style = {->,> = latex'}}
\tikzset{vertex/.style = {shape=circle,draw,minimum size=1.5em}}
[xscale=2, yscale=2, auto=left,every node/.style={circle,fill=blue!20}]
\node[vertex,accepting,initial] (n0) at (0,0) {$0$};
\node[vertex] (n1) at (4,0) {$3$};
\draw[edge,bend right=10] (n0) edge node[below] {$(2,8)$} (n1);

\draw[edge,bend right=50] (n0) edge node[below] {$(2,8)$} (n1);

\draw[edge, bend right=10] (n1) edge node[above] {$(7,1)$} (n0);

\draw[edge](n0) to[in=140,out=80, loop, style={min distance=7mm}] node[above] {$(8,2)$} (n0);
\draw[edge](n0) to[in=160,out=60, loop, style={min distance=25mm}] node[above] {$(8,2)$} (n0);

\draw[edge](n1) to[in=100,out=40, loop, style={min distance=7mm}] node[above] {$(1,7)$} (n1);

\end{tikzpicture}
\captionof{figure}{The multi-image union $\Delta_I=\Delta_1 \uplus \Delta_1 \uplus \Delta_2$ corresponding to the multiset union $C_I=C_1\uplus C_1 \uplus C_2$ of mother-graph cycles.}
\label{4_10_mi_union_2}
\end{center}
Although $\Delta_1 \uplus \Delta_1 \uplus \Delta_2$ contains the zero state and is strongly connected, we see that the indegrees and outdegrees at both vertices are unequal. Thus, the multigraph is not $L$-Eulerian, and the conditions of Corollary \ref{indegree_outdegree} are not met. Consequently, $C_1\uplus C_1 \uplus C_2$ cannot be ordered into a permutiple string.

\end{example}

\section{Summary, Conclusions, and Future Work}

In \cite{holt_6}, the question is raised regarding sufficient conditions which allow for the formation of permutiple strings. In this paper, we have provided such a condition, which leads to the following equivalence: a multiset union of mother-graph cycles can be ordered into a permutiple string if and only if the multigraph union of corresponding cycle multi-images is $L$-Eulerian (that is, if it contains the zero state, is strongly connected, and the indegree is equal to the outdegree at each vertex).

From the above considerations, we see that counting permutiples with a fixed base, multiplier, and length, and a fixed multiset of digits, is reduced to counting Eulerian circuits of a particular strongly-connected multi-image union containing the zero state having equal indegree and outdegree at each vertex. Counting Eulerian circuits of directed graphs may be accomplished in polynomial time using the BEST algorithm \cite{best}. Multigraph variations also exist \cite{farrell}. That said, if we lift the constraint of a fixed multiset of digits, then counting permutiples of a fixed base, multiplier, and length $\ell,$ becomes a problem of finding all multi-image unions which are $L$-Eulerian and have exactly $\ell$ multi-edges. We are uncertain of the difficulty of this latter problem, and it presents an area for further study.

Adding to the difficulties presented above, for larger bases and multipliers, the number of cycle multi-images increases rapidly. For example, as stated in \cite{holt_5}, there are 986 cycles of the $(4,10)$-mother graph. Thus, considering $(4,10)$-permutiples of a fixed length $\ell,$ there are a very large number of possible cycle multi-image unions to consider. Filtering these possibilities down to those which are $L$-Eulerian appears to be a difficult task, and we leave it to future efforts.

As mentioned in \cite {holt_5}, the methods presented here may be applied to the more specific {\it palintiple} problem where $\sigma$ is the reversal permutation. Finding palintiple numbers means that we only need to examine the cycle multi-images of mother-graph $1$- and $2$-cycles. However, finding palintiple strings,  which have the very particular form
\[
(d_{0},d_{k})(d_{1},d_{k-1}) \cdots (d_{j},d_{k-j}) \cdots (d_{k-j},d_{j}) \cdots (d_{k-1},d_{1})(d_{k},d_{0}),
\]
seems to be more difficult than one would expect; these are not obvious when examining the Hoey-Sloane graph or the individual cycle multi-images. This is yet another open question which could advance our understanding of palintiple numbers.

In contrast to the above, for the very specific case of palintiple numbers when $n+1$ divides $b,$ Sloane \cite{sloane} shows that the number of $k+1$-digit $(n,b)$-palintiples, for $k\geq 3,$ is $F_{\lfloor \frac{k+1}{2} \rfloor-1},$ where $F_{m}$ is the $m^{th}$ Fibonacci number. Restating this result in terms of the permutiple class $C$ considered in Example \ref{example_7}, there are $F_{\lfloor \frac{k+1}{2} \rfloor-1}$ palintiple strings of length $k+1.$ Other significant integer sequences which might arise when considering the general permutiple problem presents another open invitation for further investigation.

We conclude by presenting some questions relating to the ``derived-permutiple'' problem mentioned by \cite{holt_3,holt_4, holt_5}. A base-$b$ permutiple $(d_k,\ldots,d_0)_b$ is {\it derived} if its carries $c_j$, not including the trivial carry  $c_0=0,$ are the digits of a base-$n$ permutiple $(c_k,c_{k-1},\ldots,c_1)_n.$ This phenomenon is illustrated by the $(6,12)$-permutiple $(10,3,5,1,8,6)_{12}=6\cdot(1,8,6,10,3,5)_{12}$ whose non-trivial carries are the digits of the $(2,6)$-palintiple $(4,3,5,1,2)_6=2 \cdot (2,1,5,3,4)_6.$ If we allow for leading digits which are zero, we have another pair of examples related to one another: considering the $(2,3)$-palintiple  $(2,1,0,1)_3 = 2 \cdot (1,0,1,2)_3,$ the $(3,6)$-permutiple $(2,1,0,4,3)_6 = 3 \cdot (0,4,2,1,3)_6$ has the carry vector $(2,1,0,1,0),$ and the $(6,12)$-permutiple
$(8, 1, 6, 4, 3, 0 )_{12} = 6 \cdot (1, 4, 3, 0, 8, 6)_{12}$ has the carry vector $(2,1,0,4,3,0),$ whose non-trivial entries are the digits of the previous example. From the multigraph perspective developed here, the derived-permutiple problem may be posed as the following: given a base-$n$ permutiple $(c_k,c_{k-1},\ldots,c_1)_n,$  find an $(n,b)$-Hoey-Sloane multigraph for which the circuit $(0,c_1, \ldots, c_{k-1},c_k,0)$ may be traversed by using a collection of inputs which is a multiset union of $(n,b)$-mother-graph cycles. The above examples naturally lead to several questions: Given $(d_k,\ldots,d_0)_b,$ derived from $(c_k,c_{k-1},\ldots,c_1)_n,$ under what conditions can we find a $(b,\hat{b})$-permutiple derived from $(d_k,\ldots,d_0)_b?$ Are there cases for which this process may be continued indefinitely? Can entire permutiple classes (see Definition \ref{perm_class}) be constructed from others? What larger patterns might exist between these classes?

Other questions regarding the general problem may be found in \cite{holt_3, holt_4}.

\vskip 20pt\noindent {\bf Acknowledgement.} We are grateful to the anonymous referee for their time, as well as their valued suggestions and corrections which greatly enhanced the quality of this work.

\end{document}